\documentclass[10pt]{amsart}
\usepackage{amsmath, latexsym, amsthm, amsfonts,bm,amssymb} % Useful packages.
\usepackage{natbib} % Package for the bibliography.
\usepackage{url} % Package for embedding url.
\usepackage{enumerate}

\bibpunct{(}{)}{;}{a}{}{,} % Punctuation for citations.

\theoremstyle{plain}
\newtheorem{thm}{Theorem}

\newtheorem{lemma}{Lemma}

\newtheorem{defin}{Definition}
\newtheorem{assump}{Assumption}

\theoremstyle{remark}
\newtheorem{rem}{Remark}

\def\convd{\xrightarrow{\mathcal{D}_{\sigma_0}}}
\def\convp{\xrightarrow{\mathbb{P}_{\sigma_0}}}
\def\ex{{\rm {\mathbb{E}_{\sigma_0}\,}}}

\def\cprime{$'$}

\allowdisplaybreaks

\begin{document}

\title[Bayesian estimation for a time-inhomogeneous BM]{Consistent non-parametric Bayesian estimation for a time-inhomogeneous Brownian motion}

\author{Shota Gugushvili}
\address{Mathematical Institute\\
Leiden University\\
P.O. Box 9512\\
2300 RA Leiden\\
The Netherlands}
\email{shota.gugushvili@math.leidenuniv.nl}

\author{Peter Spreij}
\address{Korteweg-de Vries Institute for Mathematics\\
University of Amsterdam\\
PO Box 94248\\
1090 GE Amsterdam\\
The Netherlands}
\email{spreij@uva.nl}

\thanks{The research of the first author was supported by The Netherlands Organisation for Scientific Research (NWO)}

\subjclass[2000]{Primary: 62G20, Secondary: 62M05}

\keywords{Dispersion coefficient; Non-parametric Bayesian estimation; Posterior consistency; Time-inhomogeneous Brownian motion}

\begin{abstract}

We establish posterior consistency for non-parametric Bayesian estimation of the dispersion coefficient of a time-inhomogeneous Brownian motion.

\end{abstract}

\date{\today}

\maketitle

\section{Introduction}
\label{intro}

Consider a simple linear stochastic differential equation
\begin{equation}
\label{sde}
%\begin{cases}
dX_t=\sigma(t)dW_t, \quad X_0=x, \quad t\in[0,1],
%\\
%X_0=\xi.
%\end{cases}
\end{equation}
where $W$ is a Brownian motion on some given probability space and the initial condition $x$ and the square integrable dispersion coefficient $\sigma$ are deterministic. We interpret equation \eqref{sde} as a short-hand notation for the integral equation
\begin{equation*}
X_t=x+\int_0^t \sigma(s)dW_s, \quad t\in[0,1],
\end{equation*}
where the integral is the Wiener integral of $\sigma$ with respect to the Brownian motion $W.$ The process $X$ is thus a time-inhomogeneous Brownian motion. The function $\sigma$ can be viewed as a signal transmitted through a noisy channel, where the noise (modelled by the Brownian motion) is multiplicative. Note that $X$ is a Gaussian process with mean $m(t)=x$ and covariance
%\begin{equation*}
$
\rho(s,t)=\int_0^{s \wedge t} \sigma^2(u)du.
$
%\end{equation*}
By $\mathbb{P}_{\sigma}$ we will denote the law of the solution $X$ to \eqref{sde}.

Assume for simplicity that $x=0$ and denote $t_{i,n}=i/n,i=0,\ldots,n.$ Suppose that corresponding to the true dispersion coefficient $\sigma=\sigma_0,$ one has a sample $X_{t_{i,n}},i=1,\ldots,n,$ from the process $X$ at his disposal. Assuming that $\sigma_0$ belongs to some non-parametric class $\mathcal{X}$ of dispersion coefficients, our goal is to estimate $\sigma_0.$ This problem for a similar model was treated in \cite{genon92}, \cite{hoffmann97} and \cite{soulier98} using a frequentist approach. However, a non-parametric Bayesian approach to estimation of $\sigma_0$ is also possible. The likelihood corresponding to the observations $X_{t_{i,n}}$ is given by
\begin{equation}
\label{likelih}
L_n(\sigma)=\prod_{i=1}^{n} \left\{ \frac{1}{\sqrt{2\pi\int_{t_{i-1,n}}^{t_{i,n}}\sigma^2(u)du}}\psi\left( \frac{X_{t_{i,n}}-X_{t_{i-1,n}}}{\sqrt{\int_{t_{i-1,n}}^{t_{i,n}}\sigma^2(u)du}} \right) \right\},
\end{equation}
where $\psi(u)=\exp(-u^2/2).$ For a prior $\Pi$ on $\mathcal{X},$ Bayes' formula yields the posterior measure
\begin{equation*}
%\label{bayes}
\Pi(\Sigma|X_{t_{0,n}}\ldots,X_{n,n})=\frac{\int_{\Sigma} L_n(\sigma) \Pi(d\sigma)}{ \int_{\mathcal{X}} L_n(\sigma) \Pi(d\sigma) }.
\end{equation*}
of any measurable set $\Sigma\subset\mathcal{X}.$ In the Bayesian paradigm, the posterior encodes all the information required for inferential purposes. Once the posterior is available, one can proceed to computation of other quantities of interest in Bayesian statistics, such as Bayes point estimates, Bayes factors and so on.

It has been recognised since long that Bayesian procedures should be theoretically grounded through establishing posterior consistency, see e.g.\ \cite{diaconis86}. In our context posterior consistency will mean that for every neighbourhood $U_{\sigma_0}$ of $\sigma_0$ (in a suitable topology)
\begin{equation}
\label{consistency}
\Pi(U_{\sigma_0}^c|X_{t_{0,n}},\ldots,X_{t_{n,n}})\convp 0
\end{equation}
as $n\rightarrow\infty$ (the notation $ \xi_n \convp \xi$ in \eqref{consistency} and below stands for convergence of a sequence of random variables $\xi_n$ to a random variable $\xi$ in $\mathbb{P}_{\sigma_0}$-probability). In other words, a consistent Bayesian procedure asymptotically puts posterior mass equal to one on every fixed neighbourhood of the true parameter. This is similar to the study of consistency of frequentist estimators.  A method that does not appear to work in the idealised setting when an infinite amount of data is available (formalised by assuming that the sample size $n\rightarrow\infty$) should also be unattractive in the finite sample setting. Hence the importance of a study of posterior consistency. The situation is typically quite subtle in the infinite-dimensional Bayesian setting: it is known that a careless choice of the prior might render a Bayes procedure inconsistent.  For an introduction to consistency issues in Bayesian non-parametric statistics see e.g.\ \cite{ghosal99} and \cite{wasserman98}.

Our task in this work is to establish \eqref{consistency} under suitable assumptions on the class of dispersion coefficients $\sigma$ and the prior $\Pi.$ Asymptotic properties of Bayesian procedures in estimation problems for stochastic differential equations have been already considered under various setups in \cite{gugu12}, \cite{meulen0}, \cite{meulen}, \cite{panzar} and \cite{pokern}, primarily in the context of non-parametric Bayesian estimation of the drift coefficient of a stochastic differential equation. Computational approaches to non-parametric Bayesian inference for stochastic differential equations were studied in \cite{schauer} and \cite{papa12}. Convenient overviews of the available results are given in \cite{pavliotis12} and \cite{vz13}. However, in the above works dealing with Bayesian asymptotics it is assumed that either a continuous record of observations is available on the solution to a stochastic differential equation, or that the solution is observed at equispaced time points $\Delta,2\Delta,\ldots,n\Delta,$ with asymptotics treated in the latter case under the assumption that $\Delta$ is independent of $n$ and $n\rightarrow\infty.$ Our problem, on the other hand, requires a different approach due to a different sampling scheme and the fact that ergodicity of the solution to a stochastic differential equation, that played a prominent role in most of the previous works on non-parametric Bayesian approach to statistical inference for stochastic differential equations, is irrelevant in our case. Although the setup we consider looks simple, to the best of our knowledge our work is the first one to treat an inference problem for a stochastic differential equation in the so called high-frequency data case when $\Delta=\Delta_n\rightarrow 0$ as $n\rightarrow\infty$ using a non-parametric Bayesian approach. The high-frequency data setting is particularly relevant in financial mathematics, where asset prices are often modelled through stochastic differential equations and where huge amounts of observations on them separated by very short time instances are available. Perhaps the most interesting feature of the present work is the method of proof of posterior consistency, which differs in certain respects from the currently used techniques. See Section \ref{discussion} for a discussion. Also the simplicity of our model should not necessarily be considered a disadvantage: indeed, the model is somewhat similar to the Gaussian white noise model (see e.g.\ Chapter 7, \S 4 in \cite{ibragimov}), which, as is known, has triggered some important developments in mathematical statistics.

The paper is organised as follows: in the next section we formulate our main result dealing with posterior consistency for non-parametric estimation of the dispersion coefficient. Since posterior consistency is closely linked with properties of a prior $\Pi,$ in Section \ref{prior} we provide an example of a reasonable prior satisfying the assumptions made in Section \ref{main}. Section \ref{discussion} contains a brief discussion on the obtained result. The proof of our main theorem is deferred until Section \ref{proofs}, while the Appendix contains two technical lemmas used in Section \ref{proofs}.

\section{Results}
\label{main}

The non-parametric class of dispersion coefficients we will be looking at is given in the following definition.

\begin{defin}
\label{classX} Let $\mathcal{X}$ be the collection of dispersion coefficients $\sigma:[0,1]\rightarrow[\kappa,K],$ such that $\sigma\in\mathcal{X}$ is Lipschitz with Lipschitz constant $M.$ Here $0<\kappa<K<\infty$ and $0<M<\infty$ are three constants independent of the particular $\sigma\in\mathcal{X}.$
\end{defin}

\begin{rem}
\label{constrem}
Note that for a constant $\sigma$ we have $\mathbb{P}_{\sigma}=\mathbb{P}_{-\sigma}.$ A positivity assumption on $\sigma\in\mathcal{X}$ in Definition \ref{classX} can hence be viewed as a simple and natural identifiability requirement. Strict positivity assumption $\sigma\geq\kappa>0$ allows one to escape technical complications when manipulating the likelihood \eqref{likelih} (this condition has already appeared e.g.\ in \cite{hoffmann97}), while the upper bound $\sigma(t)\leq K,t\in[0,1],$ restricts the size of the non-parametric class $\mathcal{X}$ and is reasonable in light of Definition \ref{topology} given below. Finally, Lipschitz continuity of $\sigma$ comes in handy at various stages of the proof of posterior consistency. \qed
\end{rem}

The notion of posterior consistency depends on a topology on $\mathcal{X}.$ 

\begin{defin}
\label{topology}
The topology ${\mathcal{T}}$ on ${\mathcal{X}}$ is the topology induced by the $L_{2}$-norm $\|\cdot\|_{2}.$
\end{defin}

We now formalise the concept of posterior consistency.

\begin{defin}
\label{defconsistency}
Let the prior $\Pi$ be defined on $\mathcal{X}.$ We say that posterior consistency holds, if for any fixed $\sigma_0\in\mathcal{X}$ and every neighbourhood $U_{\sigma_0}$ of $\sigma_0$ in the topology ${\mathcal{T}}$ from Definition \ref{topology} we have
\begin{equation*}
\Pi(U_{\sigma_0}^c|X_{t_{0,n}}\ldots,X_{n,n}) \convp 0
\end{equation*}
as $n\rightarrow\infty.$
\end{defin}

We summarise our assumptions.

\begin{assump}
\label{standing}
Assume that
\begin{enumerate}[(a)]
\item the model \eqref{sde} is given with $x=0$ and $\sigma\in\mathcal{X},$ where $\mathcal{X}$ is defined in Definition \ref{classX},
\item $\sigma_0\in\mathcal{X}$ denotes the true dispersion coefficient,
\item a discrete-time sample $\{X_{t_{i,n}}\}$ from the solution to \eqref{sde} corresponding to $\sigma_0$ is available, where $t_{i,n}=i/n,i=0,\ldots,n.$
\end{enumerate}
\end{assump}

Let
\begin{equation*}
V_{\sigma_0,\varepsilon}=\left\{ {\sigma}\in\mathcal{X}: {\| {\sigma}-\sigma_0\|_{\infty}} <\varepsilon \right\},
\end{equation*}
where $\|\cdot\|_{\infty}$ denotes the $L_{\infty}$-norm. The following is the main result of the paper.

\begin{thm}
\label{mainthm}
Under Assumption \ref{standing} posterior consistency as in Definition \ref{defconsistency} holds, provided the prior $\Pi$ on $\mathcal{X}$ satisfies
\begin{equation}
\label{priorC}
\Pi( V_{\sigma,\varepsilon} )>0
\end{equation}
for any $\varepsilon>0$ and at any $\sigma\in\mathcal{X}.$
\end{thm}

\section{Example of a prior}
\label{prior}

In this section we provide an example of a prior satisfying condition \eqref{priorC}. Fix $0<\kappa<K<\infty$ and take a fixed Lipschitz continuous function $f:\mathbb{R}\rightarrow[0,K-\kappa]$ with Lipschitz constant $N>0$ and set $\sigma(t)=\kappa+\int_0^t f(h(s))ds,$ where $h:[0,1]\rightarrow\mathbb{R}$ ranges over the set of H\"older continuous functions of order $\beta\in(0,1/2)$ on $[0,1]$ for some fixed $\beta.$ Then each $\sigma$ maps the interval $[0,1]$ into the interval $[\kappa,K]$ and $\sigma$ is also Lipschitz with Lipschitz constant $K,$ because
\begin{equation*}
|\sigma(t)-\sigma(s)|=\left| \int_s^t f(h(u))du \right|\leq K |t-s|.
\end{equation*}
We take the collection of these functions $\sigma$ as the collection $\mathcal{X}$ from Assumption \ref{standing}~(a). We will now construct a prior $\Pi$ on $\mathcal{X}.$ Let $\widetilde{W}=(\widetilde{W}_t)_{0\leq t\leq 1}$ be a standard Brownian motion over the time interval $[0,1]$ and let $Z$ be a standard normal random variable independent of $\widetilde{W}.$ Define the Brownian motion $\overline{W}=(\overline{W}_t)_{0\leq t\leq 1}$ initialised at $Z$ by $\overline{W}_t=Z+\widetilde{W}_t$ and introduce the process $Y=(Y_t)_{0\leq t \leq 1},$ where
\begin{equation*}
Y_t=\kappa+\int_0^t f(\overline{W}_s)ds.
\end{equation*}
Our prior $\Pi$ on $\mathcal{X}$ will be the law of the process $Y.$

We have to check that the prior $\Pi$ satisfies \eqref{priorC}. To that end take a fixed $\sigma_0(t)=\kappa+\int_0^t f(h_0(s))ds$ and let $\overline{w}$ be a generic realisation of the process $\overline{W},$ so that $\overline{y}_t=\kappa+\int_0^t f(\overline{w}_s)ds$ is the corresponding generic realisation of the process $Y.$ We have
\begin{equation*}
\Pi(V_{\sigma_0,\varepsilon})=\Pi(\overline{y}:\| \overline{y}-\sigma_0 \|_{\infty}<\varepsilon)\geq \Pi\left(\overline{w}:\|\overline{w} - h_0 \|_{\infty}<\frac{\varepsilon}{N}\right),
\end{equation*}
because
\begin{equation*}
\sup_{t\in[0,1]} \left| \int_0^t [ f(\overline{w}_s) - f(h_0(s)) ]ds \right| \leq \| f(\overline{w}) - f(h_0) \|_{\infty}
\leq N \|\overline{w} - h_0 \|_{\infty}.
\end{equation*}
By Lemma 5.3 in \cite{vaart08b},
\begin{multline*}
\Pi\left(\overline{w}:\|\overline{w} - h_0 \|_{\infty}<\frac{\varepsilon}{N}\right)\\
\geq \exp\left(-\inf_{g:\|g-h_0\|_{\infty}<\varepsilon/(2N)}\frac{1}{2}\|g\|_H^2\right)
\Pi\left(\|\overline{W}\|_{\infty}<\frac{\varepsilon}{2N}\right).
\end{multline*}
Here $H$ denotes the Reproducing Kernel Hilbert Space (RKHS) of the process $\overline{W},$ $g\in H,$ while $\|\cdot\|_H$ is the RKHS norm (see \cite{vaart08b} for a detailed treatment of these concepts with a view towards non-parametric Bayesian statistics). In our case $H$ consists of absolutely continuous functions $g:[0,1]\rightarrow\mathbb{R},$ such that $\|g^{\prime}\|_2<\infty,$ while the RKHS norm is given by $\|g\|_H=\sqrt{g^2(0)+\|g^{\prime}\|_2^2},$ see p.\ 1446 in \cite{vaart08a}. Note that the set $\{g\in H:\| g - h_0 \|_{\infty} <\varepsilon/(2N) \}$ is not empty, as $h_0$ can be approximated arbitrarily closely in the $L_{\infty}$-norm by the convolution $h_0\ast k_b$ of $h_0$ with a smooth kernel $k_b(\cdot)=(1/b)k(\cdot/b),$ cf.\ p.\ 1446 in \cite{vaart08a} (we assume that $\|k^{\prime}\|_2<\infty$ and $b\rightarrow 0$). Furthermore, $\Pi(\|\overline{W}\|_{\infty}<\varepsilon/(2N))>0,$ because $\|\overline{W}\|_{\infty}$ has a strictly positive density. Condition \eqref{priorC} easily follows. In case one is interested in a smoother class of dispersion coefficients $\sigma$ than what we have just constructed, one can simply take in the above construction of the prior $\Pi$ a smoother, say $\beta$ times differentiable function $f$, and replace the Brownian motion $\overline{W}$ with a Riemann-Liouville process $R=(R_t)_{0\leq t\leq 1}$ with Hurst parameter $\beta,$
\begin{equation*}
R_t=\sum_{k=0}^{\beta} Z_k t^k+ \int_0^t (t-s)^{\beta-1/2}d \widetilde{W}_t,
\end{equation*}
where $Z_k$'s are standard normal random variables, $\widetilde{W}$ is a standard Brownian motion and $Z_0,Z_1,\ldots,Z_{\beta},\widetilde{W}$ are independent. See Section 4.2 in \cite{vaart08a} for more information on the Riemann-Liouville processes. Arguments similar to the ones given above yield that in this case as well \eqref{priorC} is satisfied.

\section{Discussion}
\label{discussion}

In the present work we established posterior consistency for a statistical model obtained from a simple linear stochastic differential equation. General techniques for proving posterior consistency for a wide range of statistical models are by now well-developed. In the i.i.d.\ setting, broadly speaking, two main approaches exist in the literature: a `classical' approach as epitomised e.g.\ by \cite{barron99} and \cite{schwartz65} (we combine these two papers into one category, because they in some sense make use of assumptions of similar type, although their actual assertions are different), and a martingale approach developed more recently in \cite{walker03} and \cite{walker04}. The first approach was extended to the  setting of independent non-identically distributed observations in \cite{choudhuri04}, see in particular Theorem A.1 there. The second approach was extended to the case of discretely observed Markov processes in \cite{ghosal06}. A general theorem for posterior consistency in \cite{choudhuri04} makes two requirements: firstly, the prior must put sufficient mass in arbitrarily small neighbourhoods of the true parameter (in an appropriate topology), and secondly, a sequence of sieves (increasing sequence of subsets of the parameter set)  guaranteeing existence of certain exponentially consistent tests has to be exhibited; see p.\ 1056 in \cite{choudhuri04} for additional details. Although in our setting the observations $X_{t_{i,n}},i=1,\ldots,n,$ are not independent, the increments $X_{t_{i,n}}-X_{t_{i-1,n}}$ are, and it appears conceivable that Theorem A.1 in \cite{choudhuri04} could be used to establish posterior consistency in our model as well. However, we opted for a different approach, see the proof of our posterior consistency result, Theorem \ref{mainthm}. A similarity shared by Theorem A.1 in \cite{choudhuri04} and Theorem \ref{mainthm} is that both theorems require that the prior puts sufficient mass in arbitrarily small neighbourhoods of the true parameter (in appropriate topologies). A difference is that due to the special structure of our model we do not need to make any reference to tests and sieves, but can establish posterior consistency by directly manipulating the posterior; see the proof of Theorem \ref{mainthm} for details. In this sense our approach to proving posterior consistency appears to be more direct and more elementary than the one that would employ Theorem A.1 in \cite{choudhuri04}. Neither do we make reference to entropy arguments as done e.g.\ in \cite{barron99}. As far as the martingale approach to posterior consistency for ergodic Markov processes is concerned, we can be brief here: ergodicity is irrelevant in our setting and in fact our special sampling scheme seems to make generalisation or modification of the arguments from \cite{ghosal06}, \cite{walker03} and \cite{walker04} impossible.

Next a brief remark on condition \eqref{priorC} on the prior $\Pi$ is in order. Although it is formulated in terms of neighbourhoods in the $L_{\infty}$-norm, the assertion returned by Theorem \ref{mainthm} employs the topology induced by the $L_2$-norm. A `discrepancy' between norms used is however not uncommon in posterior consistency results. See for instance \cite{barron99}.

\section{Proofs}
\label{proofs}

\begin{proof}[Proof of Theorem \ref{mainthm}]
Let $U_{\sigma_0}$ be an arbitrary, but fixed neighbourhood of $\sigma_0$ in the topology $\mathcal{T}$ and let $\widetilde{U}_{\sigma_0,\varepsilon}=\{\sigma\in\mathcal{X}:\| \sigma-\sigma_0 \|_2<\varepsilon\}.$ There exists $\varepsilon>0,$ such that $ \widetilde{U}_{\sigma_0,\varepsilon}\subset U_{\sigma_0},$ and hence $U_{\sigma_0}^c \subset \widetilde{U}_{\sigma_0,\varepsilon}^c.$ Fix such an $\varepsilon.$ In order to prove the theorem, it thus suffices to show that
\begin{equation}
\label{cons}
\Pi(\widetilde{U}_{\sigma_0,\varepsilon}^c|X_{t_{0,n}}\ldots,X_{t_{n,n}}) \convp 0
\end{equation}
as $n\rightarrow\infty.$ Write
\begin{equation}
\label{bayes}
\begin{split}
\Pi(\widetilde{U}_{\sigma_0,\varepsilon}^c|X_{t_{0,n}}\ldots,X_{t_{n,n}})&=\frac{\int_{\widetilde{U}_{\sigma_0,\varepsilon}^c} L_n(\sigma) \Pi(d\sigma)}{ \int_{\mathcal{X}} L_n(\sigma) \Pi(d\sigma) }\\
&=\frac{\int_{\widetilde{U}_{\sigma_0,\varepsilon}^c} R_n(\sigma) \Pi(d\sigma)}{ \int_{\mathcal{X}} R_n(\sigma) \Pi(d\sigma) },
\end{split}
\end{equation}
where $R_n(\sigma)=L_n(\sigma)/L_n(\sigma_0)$ denotes the likelihood ratio. We will separately bound the numerator and denominator on the right-hand side of the last equality in \eqref{bayes} (we will use the notation $D_n$ for the denominator and $N_n$ for the numerator) and then combine the bounds to establish \eqref{cons}. As we will see, the left-hand side of \eqref{cons} in fact decays exponentially fast to zero. Note that when establishing posterior consistency, \cite{barron99} and \cite{walker04} also treat the numerator and denominator in the expression for the posterior separately, but similarity of our approach to the one in those papers largely ends here.

Let $S_n(\sigma)=n^{-1}\log R_n(\sigma).$ Then
%\begin{equation*}
$
D_n=\int_{\mathcal{X}} \exp( n S_n(\sigma)) \Pi(d\sigma).
$
%\end{equation*}
Now
\begin{align*}
S_n(\sigma)&=\frac{1}{2}\frac{1}{n}\sum_{i=1}^{n}\log \left( \frac{\int_{t_{i-1,n}}^{t_{i,n}}\sigma_0^2(u)du}{\int_{t_{i-1,n}}^{t_{i,n}}\sigma^2(u)du} \right)\\
&-\frac{1}{2}\frac{1}{n}\sum_{i=1}^{n}\left[ \frac{( X_{t_{i,n}}-X_{t_{i-1,n}} )^2}{\int_{t_{i-1,n}}^{t_{i,n}}\sigma^2(u)du} - \frac{( X_{t_{i,n}}-X_{t_{i-1,n}} )^2}{\int_{t_{i-1,n}}^{t_{i,n}}\sigma_0^2(u)du} \right]\\
&=T_{1,n}(\sigma)+T_{2,n}(\sigma)
\end{align*}
with obvious definitions of $T_{1,n}(\sigma)$ and $T_{2,n}(\sigma).$ Let $\widetilde{\varepsilon}>0$ be a constant with its value to be chosen  appropriately later on. Since
%\begin{equation*}
$
D_n\geq \int_{V_{\sigma_0,\widetilde{\varepsilon}}}R_n(\sigma)\Pi(d\sigma),
$
%\end{equation*}
Lemmas \ref{lemma0A} and \ref{lemma1A} from the Appendix and formula \eqref{integralupb} give that with probability tending to one,
\begin{align*}
D_n &\geq \int_{V_{\sigma_0,\widetilde{\varepsilon}}} \exp\left( -\frac{4K}{\kappa^2} \widetilde{\varepsilon} n \right) \Pi(d\sigma)\\
&= \exp\left( -\frac{4K}{\kappa^2} \widetilde{\varepsilon} n \right) \Pi( V_{\sigma_0,\widetilde{\varepsilon}} ).
\end{align*}
By assumption \eqref{priorC}, $\Pi( V_{\sigma_0,\widetilde{\varepsilon}} )>0.$ Fix a constant
%\begin{equation*}
$
\beta={5K\widetilde{\varepsilon}}/{\kappa^2}.
$
%\end{equation*}
Then for all $n$ large enough,
\begin{equation*}
\exp\left(-\frac{4K}{\kappa^2}\widetilde{\varepsilon} n \right) \Pi( V_{\sigma_0,\widetilde{\varepsilon}} ) \geq e^{-\beta n}.
\end{equation*}
As a consequence, with probability tending to one,
\begin{equation}
\label{Dnbound}
D_n \geq  e^{-\beta n}
\end{equation}
as $n\rightarrow\infty.$ This is our required lower bound for the denominator $D_n.$

Using similar techniques, we will next treat the numerator $N_n.$ Firstly, note that by elementary arguments one can show that for an arbitrary fixed constant $C>0$ there exists another constant $c>0,$ such that the inequality
\begin{equation*}
\log(1+y)  \leq y -c y^2, \quad -1< y \leq C
\end{equation*}
holds (one can take $c=[2(C+1)]^{-1}$). Hence
\begin{align*}
\log \left( \frac{\int_{t_{i-1,n}}^{t_{i,n}}\sigma_0^2(u)du}{\int_{t_{i-1,n}}^{t_{i,n}}\sigma^2(u)du} \right) & \leq  \frac{\int_{t_{i-1,n}}^{t_{i,n}}[\sigma_0^2(u)-\sigma^2(u)]du}{\int_{t_{i-1,n}}^{t_{i,n}}\sigma^2(u)du}\\
&-c \left( \frac{\int_{t_{i-1,n}}^{t_{i,n}}[\sigma_0^2(u)-\sigma^2(u)]du}{\int_{t_{i-1,n}}^{t_{i,n}}\sigma^2(u)du} \right)^2
\end{align*}
for some constant $c$ independent of $\sigma\in\mathcal{X},i$ and $n.$ Therefore, after a simple, but lengthy computation employing Assumption \ref{standing} (a), cf.\ the proof of Lemma \ref{lemma0A},
\begin{align*}
T_{1,n}(\sigma) & \leq \frac{1}{2}\frac{1}{n} \sum_{i=1}^n \frac{\int_{t_{i-1,n}}^{t_{i,n}}[\sigma_0^2(u)-\sigma^2(u)]du}{\int_{t_{i-1,n}}^{t_{i,n}}\sigma^2(u)du}\\
&-\frac{c}{2}\frac{1}{n}\sum_{i=1}^n \left( \frac{\int_{t_{i-1,n}}^{t_{i,n}}[\sigma_0^2(u)-\sigma^2(u)]du}{\int_{t_{i-1,n}}^{t_{i,n}}\sigma^2(u)du} \right)^2\\
&=\frac{1}{2}\int_0^1 \frac{\sigma_0^2(u)-\sigma^2(u)}{\sigma^2(u)}du\\
&-\frac{c}{2}\int_0^1 \frac{(\sigma_0^2(u)-\sigma^2(u))^2}{\sigma^4(u)} du+O\left(\frac{1}{n}\right),
\end{align*}
where the remainder term is of order $n^{-1}$ uniformly in $\sigma\in\mathcal{X}.$ Hence
%\begin{equation}
\begin{align}
S_n(\sigma)&\leq-\frac{c}{2}\int_0^1 \frac{(\sigma_0^2(u)-\sigma^2(u))^2}{\sigma^4(u)} du \label{snsigma1}\\
&+T_{2,n}(\sigma)+\frac{1}{2}\int_0^1 \frac{\sigma_0^2(u)-\sigma^2(u)}{\sigma^2(u)}du \label{snsigma2}\\
&+O\left(\frac{1}{n}\right) \label{snsigma3}.
\end{align}
%\end{equation}
To bound from above the term on the right-hand side of inequality \eqref{snsigma1}, use the fact that
\begin{equation*}
\frac{c}{2}\int_0^1 \frac{ ( \sigma_0^2(u) - \sigma^2(u) )^2 }{\sigma^4(u)}du \geq \frac{2\kappa^2 c}{K^4}\varepsilon^2
\end{equation*}
for $\sigma\in \widetilde{U}_{\sigma_0,\varepsilon}^c .$ Furthermore, by Lemma \ref{lemma1A}, uniformly in $\sigma\in \widetilde{U}_{\sigma_0,\varepsilon}^c$ and with probability tending to one as $n\rightarrow\infty,$ the term \eqref{snsigma2} is smaller than any positive number fixed beforehand. So is the term \eqref{snsigma3}. Therefore, uniformly in $\sigma\in \widetilde{U}_{\sigma_0,\varepsilon}^c$ and with probability tending to one as $n\rightarrow\infty,$
\begin{equation*}
S_n(\sigma)  \leq -\frac{\kappa^2c}{K^4}\varepsilon^2,
\end{equation*}
say. Thus with probability tending to one as $n\rightarrow\infty,$
\begin{equation}
\label{Nnbound}
%\begin{split}
N_n=\int_{{\widetilde{U}^c_{\sigma_0,\varepsilon}}} \exp( n S_n(\sigma)) \Pi(d\sigma)
\leq  \exp\left( -\frac{\kappa^2c}{K^4}\varepsilon^2 n \right).
%\end{split}
\end{equation}
This finishes bounding the numerator $N_n.$

We now combine bounds \eqref{Dnbound} and \eqref{Nnbound} to conclude that with probability tending to one as $n\rightarrow\infty,$
\begin{equation*}
\Pi(\widetilde{U}_{\sigma_0,\varepsilon}^c|X_{t_{0,n}}\ldots,X_{t_{n,n}})\leq \exp\left( -\left(\frac{\kappa^2c}{K^4}\varepsilon^2-\frac{5K}{\kappa^2}\widetilde{\varepsilon}\right) n \right).
\end{equation*}
Picking $\widetilde{\varepsilon}$ small enough, so that 
\begin{equation*}
\frac{\kappa^2c}{K^4}\varepsilon^2-\frac{5K}{\kappa^2}\widetilde{\varepsilon}> 0,
\end{equation*}
implies \eqref{cons} and completes the proof of the theorem.
\end{proof}

\appendix
\section*{Appendix}
%\label{app}

\begin{lemma}
\label{lemma0A}
Under the same assumptions as in Theorem \ref{mainthm}, for $T_{1,n}(\sigma)$ as in the proof of Theorem \ref{mainthm}, all $\sigma\in V_{\sigma_0,\widetilde{\varepsilon}}$ simulatenously and for $n$ large enough,
%\begin{equation*}
%\label{T1n}
$
T_{1,n}(\sigma)\geq -{2K\widetilde{\varepsilon}}/{\kappa^2}.
$
%\end{equation*}
\end{lemma}

\begin{proof}
The elementary inequality
\begin{equation*}
\frac{y}{y+1}\leq\log(1+y), \quad y>-1
\end{equation*}
gives that
\begin{equation*}
\log \left( \frac{\int_{t_{i-1,n}}^{t_{i,n}}\sigma_0^2(u)du}{\int_{t_{i-1,n}}^{t_{i,n}}\sigma^2(u)du} \right) \geq  \frac{\int_{t_{i-1,n}}^{t_{i,n}}[\sigma_0^2(u)-\sigma^2(u)]du}{\int_{t_{i-1,n}}^{t_{i,n}}\sigma_0^2(u)du}. 
\end{equation*}
Next, employing Assumption \ref{standing} (a) and (c), by a simple computation one can show that
\begin{equation*}
\frac{\int_{t_{i-1,n}}^{t_{i,n}}[\sigma_0^2(u)-\sigma^2(u)]du}{\int_{t_{i-1,n}}^{t_{i,n}}\sigma_0^2(u)du}
=\frac{\sigma_0^2(t_{i-1,n})-\sigma^2(t_{i-1,n})}{\sigma_0^2(t_{i-1,n})}+O\left(\frac{1}{n}\right),
\end{equation*}
where the remainder term is of order $n^{-1}$ uniformly in $\sigma\in\mathcal{X}.$ Therefore,
\begin{equation*}
T_{1,n}(\sigma)\geq \frac{1}{2}\frac{1}{n}\sum_{i=1}^n \frac{\sigma_0^2(t_{i-1,n})-\sigma^2(t_{i-1,n})}{\sigma_0^2(t_{i-1,n})}+O\left(\frac{1}{n}\right).
\end{equation*}
By another simple computation,
\begin{equation*}
\frac{1}{n}\sum_{i=1}^n \frac{\sigma_0^2(t_{i-1,n})-\sigma^2(t_{i-1,n})}{\sigma_0^2(t_{i-1,n})}=\int_0^1 \frac{\sigma_0^2(u)-\sigma^2(u)}{\sigma_0^2(u)}du+O\left(\frac{1}{n}\right),
\end{equation*}
where the remainder term is of order $n^{-1}$ uniformly in $\sigma\in\mathcal{X}.$
For $\sigma\in V_{\sigma_0,\widetilde{\varepsilon}}$ we have under Assumption \ref{standing} (a) that
\begin{equation}
\label{integralupb}
\left|\int_0^1 \frac{\sigma_0^2(u)-\sigma^2(u)}{\sigma_0^2(u)}du\right| \leq \frac{2K}{\kappa^2}\widetilde{\varepsilon}.
\end{equation}
This implies the statement of the lemma.
\end{proof}

\begin{lemma}
\label{lemma1A}
Denote
\begin{equation*}
Q_n(\sigma)=\left|T_{2,n}(\sigma)+\frac{1}{2}\int_0^1 \frac{\sigma_0^2(u)-\sigma^2(u)}{\sigma^2(u)}du\right|,
\end{equation*}
where $T_{2,n}(\sigma)$ is defined in the proof of Theorem \ref{mainthm}. Then under the same assumptions as in Theorem \ref{mainthm},
$\sup_{\sigma\in\mathcal{X}}Q_n(\sigma)\convp 0$
as $n\rightarrow\infty.$ Furthermore, for any fixed $\widetilde{\varepsilon}>0,$ for all $\sigma\in V_{\sigma_0,\widetilde{\varepsilon}}$ simultaneously, with probability tending to one as $n\rightarrow\infty,$
%\begin{equation*}
%\label{T2n}
$
T_{2,n}(\sigma) \geq -{2K{\widetilde{\varepsilon}}}/{\kappa^2}.
$
%\end{equation*}
\end{lemma}

\begin{proof}
The first statement of the lemma will be derived from an application of Theorem 18.14 in \cite{vaart98}. In particular, viewing $Q_n$ as a process on $\mathcal{X}$ with bounded sample paths, we will show that it converges in distribution to a zero process on $\mathcal{X}.$ The first statement of the lemma will then be a consequence of equivalence of convergence in distribution and in probability for constant limits. Note that in order to circumvent possible (non)-measurability issues, outer probability is employed in the formulation of Theorem 18.14 in \cite{vaart98} (see Section 18.2 in \cite{vaart98} for more information on outer probability). Since no such problems will arise in our setting, we can instead directly work under probability $\mathbb{P}_{\sigma_0}.$ Indeed, the summands in $T_{2,n}(\sigma)$ are of the form $F_{i,n}(\sigma)(X_{t_{i,n}}-X_{t_{i-1,n}})^2$, with $F_{i,n}$ the obvious functional of $\sigma$. Hence taking the supremum over $\sigma$ does not affect the measurability property of $T_{2,n}(\sigma).$

In order to apply Theorem 18.14 from \cite{vaart98}, we need to verify its conditions. In our setting they reduce to the following ones: firstly, marginal vectors of $Q_n$ must converge in distribution to zero vectors, i.e.\
\begin{equation}
\label{eq10}
(Q_n(\sigma_1),\ldots,Q_n(\sigma_{\ell}))\convd (\underbrace{0,\ldots,0}_{\ell}), \quad \ell\in\mathbb{N}.
\end{equation}
Secondly, the tightness condition must be satisfied: for arbitrary constants $\eta>0$ and $\xi>0,$ one must be able to find a partition of $\mathcal{X}$ into finitely many $\mathcal{X}_1,\ldots,\mathcal{X}_{\ell},$ such that
\begin{equation}
\label{eq11}
\limsup_{n\rightarrow\infty}\mathbb{P}_{\sigma_0}\left( \sup_{1\leq k\leq \ell}\sup_{\sigma_1,\sigma_2\in\mathcal{X}_k} |Q_n(\sigma_1) - Q_n(\sigma_2)| \geq \xi \right)\leq \eta.
\end{equation}

Denote 
\begin{equation*}
\mathcal{F}_{i,n}=\sigma(X_{t_j,n},j=1,\ldots,i)
\end{equation*}
and
\begin{equation*}
\chi_{i,n}(\sigma)=-\frac{1}{2}\frac{1}{n}( X_{t_i,n} - X_{t_{i-1,n}} )^2 \frac{ \int_{t_{i-1,n}}^{t_{i,n}} [\sigma_0^2(u) -\sigma^2(u)]du }{ \int_{t_{i-1,n}}^{t_{i,n}} \sigma_0^2(u) du \int_{t_{i-1,n}}^{t_{i,n}} \sigma^2 (u)du }.
\end{equation*}
Also let $\ex$ be the expectation operator with respect to measure $\mathbb{P}_{\sigma_0}.$ Note that
\begin{align*}
\sum_{i=1}^n \ex(\chi_{i,n}(\sigma)|\mathcal{F}_{i-1,n})&=-\frac{1}{2}\frac{1}{n}\sum_{i=1}^n \frac{ \int_{t_{i-1,n}}^{t_{i,n}} [\sigma_0^2(u) -\sigma^2(u)]du }{ \int_{t_{i-1,n}}^{t_{i,n}} \sigma^2 (u)du }\\
&=-\frac{1}{2}\int_0^1 \frac{\sigma_0^2(u)-\sigma^2(u)}{\sigma^2(u)}du+O\left(\frac{1}{n}\right),
\end{align*}
where the remainder term is of order $n^{-1}$ uniformly in $\sigma\in\mathcal{X}.$ Furthermore, Assumption \ref{standing} (a) yields that
\begin{equation*}
\ex(\chi_{i,n}^2(\sigma)|\mathcal{F}_{i-1,n})=\frac{3}{4n^{2}}\left( \frac{ \int_{t_{i-1,n}}^{t_{i,n}} [\sigma_0^2(u) -\sigma^2(u)]du }{ \int_{t_{i-1,n}}^{t_{i,n}} \sigma^2 (u)du } \right)^2=O\left(\frac{1}{n^{2}}\right),
\end{equation*}
where the order bound is uniform in $\sigma\in\mathcal{X}.$ It follows that
\begin{equation*}
\sum_{i=1}^n \ex(\chi_{i,n}^2(\sigma)|\mathcal{F}_{i-1,n})\rightarrow 0.
\end{equation*}
Lemma 9 in \cite{genon93} then implies that $Q_n(\sigma)\convp 0.$  This verifies \eqref{eq10}.

We will now check \eqref{eq11}. Fix $\xi$ and $\eta$ in \eqref{eq11}. By a lengthy, but simple computation employing Assumption \ref{standing} (a) and the triangle inequality,
\begin{equation}
\label{eq*}
 |Q_n(\sigma_1) - Q_n(\sigma_2)|\\
 \leq \frac{K}{\kappa^4}\|\sigma_1-\sigma_2\|_{\infty}\sum_{i=1}^n {(X_{t_{i,n}}-X_{t_{i-1},n})^2}+\frac{K^3}{\kappa^4}\|\sigma_1-\sigma_2\|_{\infty}.
\end{equation}
By the Arzel\`a-Ascoli theorem, under Assumption \ref{standing} (a) the family $\mathcal{X}$ is totally bounded for the supremum metric. By definition this means that for every $\zeta>0$ there exists a finite set $\widetilde{\mathcal{X}}\subset\mathcal{X},$ such that for any $\sigma\in\mathcal{X}$ there is some $\widetilde{\sigma}\in\widetilde{\mathcal{X}}$ with $\|\sigma-\widetilde{\sigma}\|_{\infty}<\zeta/2.$ This and the triangle inequality imply existence of a finite partition $\mathcal{X}_1,\ldots,\mathcal{X}_{\ell}$ of $\mathcal{X},$ such that
\begin{equation}
\label{eq**}
\sup_{1\leq k\leq\ell}\sup_{\sigma_1,\sigma_2\in\mathcal{X}_k} \| \sigma_1-\sigma_2 \|_{\infty}<\zeta.
\end{equation}
Furthermore, by the definition of the quadratic variation of the process $X,$
\begin{equation}
\label{eq***}
\sum_{i=1}^n {(X_{t_{i,n}}-X_{t_{i-1},n})^2}\convp \int_0^1 \sigma_0^2(u)du.
\end{equation}
Combination of \eqref{eq*}--\eqref{eq***} yields \eqref{eq11} for $\zeta$ small enough,
and consequently the first statement of the lemma too. The second statement of the lemma is a consequence of the first one, the fact that $\sigma\in V_{\sigma_0,\widetilde{\varepsilon}},$ an analogue of inequality \eqref{integralupb},
\begin{equation*}
\left| \frac{1}{2} \int_0^1 \frac{ \sigma_0^2(u) - \sigma^2(u) }{ \sigma^2(u) }du \right| \leq \frac{K}{\kappa^2}\widetilde{\varepsilon},
\end{equation*}
and of a simple rearrangement
\begin{align*}
T_{2,n}(\sigma)&=T_{2,n}(\sigma)+\frac{1}{2}\int_0^1 \frac{\sigma_0^2(u)-\sigma^2(u)}{\sigma^2(u)}du\\
&-\frac{1}{2}\int_0^1 \frac{\sigma_0^2(u)-\sigma^2(u)}{\sigma^2(u)}du.
\end{align*}
This completes the proof of the lemma.
\end{proof}

\bibliographystyle{plainnat}

\end{document}